\documentclass[12pt]{amsart}

\usepackage{url}

\usepackage{amssymb}
\usepackage{amsmath}
\usepackage{amsthm}
\newcounter{nushka}
\newtheorem{thm}{Theorem}

\newtheorem{lem}[nushka]{Lemma}

\newtheorem{cor}{Corollary}

\theoremstyle{definition}

\newcommand{\R}{\mathbb R}
\newcommand{\C}{\mathbb C}

\newcommand{\e}{\varepsilon}

\newcommand{\D}{\mathbb D}
\newcommand{\T}{\mathbb T}
\newcommand{\cT}{\mathcal{T}}

\newcommand{\cL}{\mathcal L}
\newcommand{\cE}{\mathcal E}

\newcommand{\wt}{\widetilde}

\renewcommand{\ge}{\geqslant}
\renewcommand{\le}{\leqslant}

\newcommand{\rest}{{\lfloor}}
\newcommand{\vf}{\varphi}

\newcommand{\fz}{{\mathfrak{z}}}

\newcommand{\z}{\zeta}

\begin{document}

\title[Mean type and Martin functions]{Mean type of functions of bounded characteristic and Martin functions in Denjoy domains}

\author{Alexander Volberg}
\address{Alexander Volberg, Department of Mathematics,  Michigan State University, East Lansing, Michigan, USA}
\thanks{A. V. was supported by NSF grant DMS 1301579}

\author{Peter Yuditskii}
\address{Peter Yuditskii, Abteilung der Dynamische System, Johannes Kepler Universit\"at, Linz, Austria}
 \thanks{P.Yu. was supported by the Austrian Science Fund, project no: P25591-N25}
\date{\today}

\begin{abstract}
Functions of bounded characteristic in simply connected domains have a classical factorization to Blaschke, outer and singular inner parts. The latter has a singular measure on the boundary assigned to it. The exponential speed
of change of a function when approaching a point of a boundary (mean type) corresponds to a point mass at this point.
In this paper we consider the analogous relation for functions in arbitrary infinitely connected (Denjoy) domains. The factorization result holds of course with one important addition: all functions involved become multiple valued even though the initial function was single valued. The mean type now can be measured by using the Martin function of the domain. But this result does not follow from the lifting to the universal covering of the domain because of the simple (but interesting) reason that the mean types of the original and the lifted functions can be completely different. 
\end{abstract}

\maketitle

\section{Introduction and Motivation}
\label{Intro}

In the theory of de Branges \cite{dB} the main objects are various classes of meromorphic functions $f$ in $\C$ such that $f\rest \C_+\cup\C_-$ is a function of bounded characteristic in $\C_+$ and in $\C_-$. Recall that functions of bounded characteristic in, say, $\C_+$, are just the ratios of two bounded analytic functions in $\C_+$, $\vf_1/\vf_2$, where $\vf_2$ is not identically zero.
The most famous such class is the class of Cartwright \cite{Lev}.

The following number is very important in de Branges' theory of entire function \cite{dB}:
\begin{equation}
\label{mt}
h=\limsup_{|z|\rightarrow \infty, z\in \C_+\cup\C_-} \frac{\log |f(z)|}{|z|}\,.
\end{equation}
This is, the so-called,  {\it mean type} of the given function.
This parameter plays a major part in the spectral theory of Dirac operators and $1D$ Schr\"odinger operators. Its precise role in the general theory of canonical systems was clarified by de Branges, see \cite[Section 39]{dB}. In fact, de Branges has shown the necessity of subtler grading parameter than the one given by the mean type. It is feasible that the growth in the scale of Martin function in Denjoy domains (discussed in the present article)  might serve as this subtler parameter.

Indeed, in the spectral theory of differential operators, in the theory of canonical systems, and in approximation theory one routinely encounters  more general classes of meromorphic or entire functions \cite{Yu}. Namely, let $E$ be a closed subset of $\R$, $\infty\in E$, and $\Omega= \C\setminus E$.
Then the class $B_E$ consists of   entire functions $f$ whose restriction to $\Omega$ are functions of bounded characteristic in $\Omega$. By that one should understand the following. Let $\pi:\C_+\rightarrow \Omega$ be a universal covering, $f$ is called a function of bounded characteristic in $\Omega$ if its lifting $f\circ \pi$ is  a function of bounded characteristic on $\C_+$.

Domain in $\C$ whose boundary belongs to $\R$ is called {\it Denjoy domain},  analytic properies of Denjoy domains were widely studied. Our $\Omega$ is a Denjoy domain. In Denjoy domain $\Omega$, $\infty\in \partial \Omega$, one can consider an important positive harmonic function called {\it the Martin function at infinity}.

In what follows we consider only sets $E\subset \R$ whose every point is Dirichlet regular.

{\bf Definition.}
Let $\Omega=\C\setminus E$, the symmetric Martin function $M$ at infinity is a non-zero positive harmonic function in $\Omega$ such that $M\rest E=0$ and $M(\bar z)=M(z)$.

The set of positive harmonic functions vanishing on $E$ form a one dimensional cone or a two dimensional cone, see \cite{Be}. 
In this article we will always have the situation that this cone is one dimensional, that is any Martin function is collinear to the symmetric one.

If one considers the cone of all positive harmonic function in a domain, its extremal rays form the so-called Martin boundary of the domain. 
One can read about Martin functions and Martin boundary in \cite{Be},\cite{Ha},  \cite{EYu} for example. 

{\bf Definition.} By symbol $B_E(h)$ we  denote the class of entire functions such that 1) they are of bounded characteristic (see above) in $\Omega=\C\setminus E$, that is $f\in N(\Omega)$, and
$$
2)\,\,\,\limsup_{|z|\rightarrow \infty, z\in \Omega} \frac{\log |f(z)|}{M(z)}\le h\,.
$$

We formulate now a factorization theorem for $f\in B_E(h)$. Analytic functions  $f\in N(\Omega)$, that is of bounded characteristsic in $\Omega$, have inner/outer factorization very similar to such classical factorization in the disc or half-plane \cite{Ne}, \cite{Ga}. Outer part of $F$ is $F_{out}$ such that
$$
\log F_{out}(z) = u+i\wt u\,,\,\, \,z\in \Omega,
$$
where
$$
u(z)= \int_E \log |F(t)| d\omega_\Omega(t,z),
$$ $\omega(t, z)=\omega_\Omega(t,z)$ denotes the harmonic measure of $\Omega$ with pole at $z\in \Omega$, and $\wt u$ is a harmonic conjugate of $u$ in $\Omega$.

Generically these functions $\wt u$ and $F_{out}$ are not single valued.

Blaschke products in $\Omega$ are defined as follows. Let $\fz$ be a universal covering of $\Omega$ by the disc $\D$, let $\Gamma$ be a Fuchsian group of this covering: $\Omega=\D/\Gamma$. Blaschke factor in $\Omega$ is defined by fixing a point $\z_0\in \D$, $z_0:=\fz(\z_0)$, and considering
$$
b_{z_0}(\z):= \Pi_{ \gamma\in \Gamma}\frac{\zeta-\gamma(\z_0)}{1-\overline{\gamma(\z_0)}\z}\frac{\overline{\gamma(\z_0)}}{|\gamma(\z_0)|}\,.
$$

Denoting by $G(z, z_0)$ Green's function of $\Omega$, we can write
$$
G(\fz(\z), z_0) =\log \frac{1}{|b_{z_0}(\z)|}\,.
$$

We introduce now multi-valued character  automorphic holomorphic function 
$$
\Phi_{z_0}(z) := e^{- G(z, z_0) -i \wt{G(z, z_0)}}\,.
$$
It is sometimes called {\it complex Green's function} of $\Omega$ with pole at $z_0$.
Now we can write 
$$
\Phi_{z_0}(\fz(\z)) = b_{z_0}(\z)\,.
$$

{\bf Definition.} Blaschke product in $\Omega$ is $B_\Omega(z):= \Pi_{k=1}^\infty \Phi_{z_k}(z)$, where the Blaschke condition in $\Omega$ is satisfied: 
$\sum_{k=1}^\infty G(0, z_k) <\infty$.

\medskip

The reader awaits for us to introduce the notion of singular inner function in $\Omega$. This of course can and need be done if we want to have inner/outer factorization of functions of $N(\Omega)$. One can introduce singular inner functions in $\Omega$ by integrating Martin functions over  Martin boundary. But let us recall the reader, that our functions will be also entire functions. This extra assumption easily implies that the singular inner function can be only of a very simple type, it can only involve point masses in the points of Martin boundary that lie over infinity. In our case we will have $E\subset \R_-$, and so (\cite{Be}) there will be only one Martin function (up to a positive multiplicative constant) at infinity. So all possible singular inner function involved in the factorization of functions from $B_E(h)$ can be only  of the following form $e^{a(M+i\wt M)}, \,a\in \R$. Now we can state the following factorization theorem.

\begin{thm}[inner/outer factorization of entire functions of bounded characteristic in $\Omega$]
\label{fact}
Let $f\in B_E(h)$. Then
$$
f= B_\Omega \cdot f_{out}\cdot e^{a(M+i\wt M)},
$$
where $h\le a$.
\end{thm}

Notice that our resoning preceding the statement of this theorem has already almost proved it, and also proved the following correpondence between the factorization of the lifted function $f\circ \fz(\z)$ and the factorization of the original function $f\in B_E(h)$: if $F:= f\circ \fz$, and $F=B\cdot F_{out} \cdot S$ is its factorization to 
a Blaschke product, an outer part, and a singular inner part, then
$$
B = B_\Omega\circ \fz,\, F_{out}= f_{out} \circ \fz,\, S= e^{a(M+i\wt M)}\circ \fz.
$$
But the inequality 
\begin{equation}
\label{ah}
h\le a
\end{equation} 
has not been proved. This ``small" claim will be one of the main results of the present paper. We also indicate cases when $a=h$.

Let us explain the difficulty behind this seemingly trivial statement. For this purpose it is more convenient to have universal covering $\pi:\C_+\rightarrow \Omega$, $\pi(\infty)=\infty$. Then given $f\in B_E(h)$ and $F:=f\circ \pi$ we can factorize $F$, $F= B\cdot F_{out} \cdot S$. Here $S$ is a singular inner function in $\C_+$, in particular
\begin{equation}
\label{PC}
u(z)= Ay + \int_{-\infty}^\infty \frac{y}{(x-t)^2+ y^2} d\sigma_s(t)\,,
\end{equation}
where $\sigma_s$ is a singular measure on $\R$. The convenience of $\C_+$ in comparison with $\D$ is explained by the special role of $\infty$ (of course we could have had the same role assigned to point $1\in\partial \D$). 

Recall that the symmetric Martin function $M(z)$ is defined up to a positive multiplier.

{\bf Question 1.} Is it possible to normalize $M(z)$ such that
 $A$ in \eqref{PC} is equal to $a$ in Theorem \ref{fact}?

Clearly, measure $A\delta_\infty  + d\sigma_s$ appeared just from lifting of $a\delta_\infty^E$. By this we mean that we have $a\delta_\infty^E$ measure at (unique) infinity point of the Martin boundary and it defines  our singular inner function $e^{a(M+i\wt M)}$ in $\Omega$. Then we lift this singular inner function $e^{a(M+i\wt M)}$  into $\C_+$ by $S:=e^{a(M+i\wt M)}\circ \pi$, and then we consider the Herglotz representation of $\log\frac1{|S|}$ by singular (it is necessarily singular by the principle of correspondence of harmonic measures under lifting) measure $A\delta_\infty + d\sigma_s$.  The same can be done of course by lifting into the disc $\D$ as follows $S:=e^{a(M+i\wt M)}\circ \fz$.

{\bf Definition.}  Singular measure $A\delta_\infty+ d\sigma_s$ (or corresponding singular measure on $\partial\D$) will be called the {\it lifting} of delta measure at infinity under lifting $\pi:\C_+\rightarrow \Omega$  (correspondingly, under lifting $\fz: \D\rightarrow \Omega$, we have $A\delta_1+ d\sigma_s$).

Here is a related very natural question.

{\bf Question 2.} Can the lifting of a point measure at infinity be purely continuous singular measure?

In fact, we will show that $A$ may not have anyting to do with $a$, for example, it can be that $A=0, a\neq 0$. Also we will show that the lifting of the point mass can be 1) purely point, 2) purely singular continuous measure, 3) but never the mixture of two. The answer depends on how complicated $E$ is  at infinity.

We can now explain the reason why the inequality $h\le a$ in Theorem \ref{fact} is far from being trivial.
By definition of $h$ we have
$$
h=\limsup_{z\to\infty}\frac{|f_{out}||B_\Omega|\log|e^{a(M+i\wt M)}|}{M(z)}
 \le \lim_{z\to\infty} \frac{\log|e^{a(M+i\wt M)}|}{M(z)}+
$$
$$
 \limsup_{z\to\infty}\frac{\log|f_{out}(z)|}{M(z)} +\limsup_{z\to\infty}\frac{\log|B_\Omega(z)|}{M(z)}\le
$$
$$
a +\limsup_{z\to\infty}\frac{\log|f_{out}(z)|}{M(z)} +\limsup_{z\to\infty}\frac{\log|B_\Omega(z)|}{M(z)}.
$$

To prove $h\le a$  one should  prove the theorem:

\begin{thm} 
\label{main1}
Let $E\subset \R_-$. Then 
\begin{equation}
\label{out}
\lim_{z\rightarrow\infty, z\in \Omega} \frac{\log |f_{out}(z)|}{M(z)} = 0\,.
\end{equation}
\end{thm}
Of course, the  inequality
$$
\limsup_{z\rightarrow\infty, z\in \Omega} \frac{\log |B_\Omega(z)|}{M(z)} \le 0\,,
$$
 is obvious because $|B_\Omega|\le1$.
In fact, we believe that in a special sense (but not literally)
\begin{equation}
\label{B0}
\lim_{z\rightarrow\infty, z\in \Omega} \frac{\log |B_{\Omega}(z)|}{M(z)} = 0\,.
\end{equation}
We will prove \eqref{B0} in some special cases.

\begin{thm} 
\label{main2}
Let $E\subset \R_-$. Then 
\begin{equation}
\label{out}
\lim_{z\rightarrow\infty, z\in \Omega} \frac{\log |B_\Omega(z)|}{M(z)} = 0\,,
\end{equation}
if all zeros of $B_\Omega$ lie in $\R_-\setminus E$.
\end{thm}

In this case we have equality $a=h$ just because
$$
a=\lim_{z\to\infty} \frac{\log|e^{a(M+i\wt M)}|}{M(z)}\le \limsup_{z\to\infty}\frac{|f_{out}||B_\Omega|\log|e^{a(M+i\wt M)}|}{M(z)}
-
$$
$$
- \liminf_{z\to\infty}\frac{\log|f_{out}(z)|}{M(z)}+\liminf_{z\to\infty}\frac{\log|B_\Omega(z)|}{M(z)}=h.
$$

\bigskip

Now we want to look at the relationship between numbers $a$ and $A$. Number $A$ can clearly be calculated as follows using the lifting of $f$ into half-plane $\C_+$, $F(\z):=f \circ \pi (\z)$,
$$
A=\limsup_{\z\in \C_+, \z\to\infty}\frac{\log |F(\z)|}{|\z|}.
$$

\begin{lem}
\label{easy1}
Let $E\subset \R_-$. Function $M\circ \pi$ is a Poisson transform of pure point mass  (lifting of delta measure at infinity is pure point) if and only if
\begin{equation}
\label{Muu}
\lim_{x\to+\infty} \frac{M(x)}{U(x)} >0,
\end{equation}
where $V(z)+iU(z)$ is the branch of $\pi^{-1}\rest\C_+$.
\end{lem}

\begin{proof}
 Function $M\circ \pi$ is a positive harmonic function in $\C_+$. By the principle of preservation of harmonic measure, its a. e. non-tangential limits are zero. So it is given by the Poisson integral of positive singular measure $b\delta_\infty + \mu_s$. It is well-known that $b= \lim_{\eta\to+\infty}\frac{M\circ\pi(i\eta)}{y\eta}$. We change variable in this equality: $\xi+i\eta=\z \pi^{-1} (x+iy)=\pi^{-1}(z)=V(z)+iU(z)$. Here we put $V(z)+iU(z)$ to be  the branch of $\pi^{-1}\rest\C_+$ mapping $\C_+$ onto the fundamental domain of a corresponding Fuchsian group $\Gamma$ in $\C_+$ such that $\C_+/\Gamma= \Omega$. In particular,
the right half-axis is mapped into $\{i\eta, \eta>0\}$. Then we see that $b>0$ if and only if \eqref{Muu} is satisfied.

Now let us prove that if $M\circ\pi$ has a nontrivial point mass in its measure, then this measure is pure point. If $x\in [-\infty,\infty]$ caries a point mass than every point of the orbit $orb(x)$ under the Fuchsian group $\Gamma$ carries point masses of $M\circ \pi$. Corresponding point measure is denoted by $\nu_s$, and if we build its Poisson extension $p(\z)$ it will be automatically $\Gamma$-invariant. This means that $p$ is a lifting, namely, 
$p(\z)=P\circ\pi(\z)$, where $P$ is a positive harmonic function in $\Omega$. Moreover, as $\nu_s$ is ``only a part" of singular measure of $M\circ \pi$, we get $P\circ \pi=p\le M\circ \pi$. So $P\le M$. But $M$ is Martin's function, so $P= c M$. But then $P\circ \pi= c M\circ \pi$. This means that the  measure of $M\circ \pi$ is proportional to the measure of $P\circ \pi$. This latter being pure point, we conclude that measure of $M\circ \pi$ must also be pure point.
\end{proof}

\bigskip

This simple lemma  ``answers" the question when $A\neq 0$ given that $a\neq 0$.

We can recall that for functions $F$ of bounded characteristic in $\C_+$, and for such functions this $\limsup_{\z\in \C_+, \z\to\infty}\frac{\log |F(\z)|}{|\z|}$ can be calculated via  the factorization of $F$ into outer, Blaschke product, and singular inner functions.  Namely, this $\limsup$ is exactly the mass  at infinity of the measure of the singular inner part of the factorization of $F$.

The singular inner part of $F$ was denoted by $S$ and we saw that $S=e^{a(M+i\wt M)}\circ \pi$. Hence,
$A\neq 0$ if and only if $a\cdot\limsup_{\z\to\infty}\frac{M\circ \pi(\z)}{|\z|}\neq 0$. Changing the variable $\z=\pi^{-1}(z)$, we conclude this happens if and only if \eqref{Muu} happens.

\subsection{Comb domains}\label{comb} Let $E$ be a union of a finite number of intervals,
$$
E=[b_0,a_0]\setminus\cup_{j=1}^g(a_j,b_j).
$$
Let us associate to these data the hyperelliptic Riemann surface
$$
\mathcal R=\{(z,s): s^2=\prod_{j=0}^g(z-a_j)(z-b_j)\}.
$$
One can visualize this surface as the two copies of the domain $\Omega=\bar \C\setminus E$ glued in an appropriate way along the set $E$. 

The Abelian integrals on this surface have a very explicit form. In particular, the Abelian integral of the third kind with logarithmic singularities at $(z_0,\pm s_0)$, $z_0\in \R\setminus [b_0,a_0]$, is given by
\begin{equation*}\label{abe}
H(z,z_0)=\int_{a_0}^z\frac{P_g(\lambda)}{\lambda-z_0}\frac{s_0 d\lambda}{s(\lambda)}
\end{equation*}
where $P_g(z)$ is a polynomial of degree $g$ such that $P_g(z_0)=1$. If we fix the remaining coefficients of $P_g$ by the conditions
\begin{equation}\label{periods}
\int_{a_j}^{b_j}\frac{P_g(\lambda)}{\lambda-z_0}\frac{s_0 d\lambda}{s(\lambda)}=0,\ 1\le j\le g,
\end{equation}
we obtain the standard representation for the Green function in $\Omega$
\begin{equation}\label{htog}
H(z,z_0)=G(z,z_0)+i\wt{G(z,z_0)}.
\end{equation}
Note that the conditions \eqref{periods} imply that $P_g$ has exactly one zero $\lambda_j$ in each gap $(a_j,b_j)$. Thus 
$$
H(z,z_0)=\int_{a_0}^z\frac{\prod_{j=1}^g(\lambda-\lambda_j)}{\prod_{j=1}^g(z_0-\lambda_j)}\frac{s(z_0)d\lambda}{(\lambda-z_0)s(\lambda)}
$$
 represents a Schwarz-Christoffel integral. That is, $\Theta(z):=iH(z,z_0)$  maps the upper half-plane onto a special polygon, a so-called comb domain,
 $$
 \Pi=\{u+iv: -\pi<u<0,\ v>0\}\setminus\{-\pi\omega_k+i v: 0<v\le h_k, 1\le k\le g\},
 $$
 where $\omega_k\not=\omega_j$ for $k\not=j$ and $h_k>0$.
 
 This construction one can widely generalize.
 
 \begin{thm}[\cite{EYu}]
 Let 
 $$
 \Pi=\{u+iv: \ -\pi<u<0,\ v\ge h(u)\},
 $$
 where $h$ is a non-negative upper semi-continuous function bounded from above and equal $0$ almost everywhere with respect to the Lebesgue measure on $[-\pi,0]$. For every such $\Pi$ the conformal mapping $\Theta:\C_+\to \Pi$,  normalized by
 $$
 \Theta (z_0)=\infty, \Theta (a_0)=i\sup_{u\to 0}{h(u)}, \ \Theta(b_0)=-\pi+i\sup_{u\to-\pi} h(u),
 $$
 is related with the Green function   $G(z,z_0)$ of some closed set $E$ by the formula $G(z,z_0)=\Im \Theta(z)$, $\Im z> 0$. 

If $h(u)$ does not vanish only on a countable set $\{-\pi\omega_k\}^\infty_{k\ge1}$ and $\lim_{k\to\infty} h(-\pi\omega_k)=0$, then $\Theta(z)$ can be extend by continuity on $\R$, that is, $E$ is Dirichlet regular.
 \end{thm}
 
 The following three possible kind of comb domains arise from symmetric Martin functions.

 \begin{itemize}
 \item[Case A.]
 $
 \Pi=\{u+iv: \ -\infty<u<\infty,\ v\ge h(u)\}.
 $
 \item[Case B.]
 $
 \Pi=\{u+iv: \ u_-<u<u_+,\ v\ge h(u)\},
 $
 where $u_-=-\infty$, $u_+<\infty$ or $u_+=\infty$, $u_->-\infty$.
  \item[Case C.]
  $
 \Pi=\{u+iv: \ -\pi<u<0,\ v\ge h(u)\}
 $
 and $\sup_{u\to 0} h(u)=\infty$ and/or $\sup_{u\to -\pi} h(u)=\infty$.
 \end{itemize}
As before, in all three cases $h$ is a non-negative upper semi-continuous function bounded from above and equal $0$ almost everywhere with respect to the Lebesgue measure. Also $\Theta(\infty)=\infty$ and $M(z)=\Im \Theta(z)$.

Note that in a finite gap case only Cases A and B are possible and the corresponding conformal mapping is related to Abelian integrals of the second type with a pole at an end of the gap (Case B) or in an internal point on $E$ (Case A).

Finally, we define one more sort of combs related to the Abelian integrals with poles of the second order in an internal point $z_0\in E$, that is, $z_0\not=a_j$ and $z_0\not=b_j$. 

Consider $\C\setminus \R_+$ with a system of slits
$$
\{\omega^+_k+iy,\ 0<y\le h^+_k\}_{k\ge 0}\quad\text{and}\quad \{\omega^-_k-iy,\ 0<y\le h^-_k\}_{k\ge 0},
$$
where $\omega_0^\pm=0$ and $\omega_k^\pm>0$, $k\ge 1$. The conformal map $\Theta$ is normalized by the conditions
$$
\Theta(z_0)=\infty,\ \Theta(a_0)=+0-i0, \ \Theta(b_0)=+0+i0.
$$
This map allows us to define a \textit{special path} $\{\Theta(x)\in\C_+\}_{x\in \R_-}$ to an \textit{arbitrary} point $z_0=\lim_{x\to-\infty}\Theta(x)$ on $E$, which is not an end point of a gap $(a_j,b_j)$.
Note that in the general setting for this comb certain modifications like Cases B and C above are also possibly required.

We will need the representation of Green's function as Abel integral in the last section.

\section{When $M\circ\fz$ has a point mass representation}
\label{Mpoint}

Let $\fz:\D \rightarrow \Omega$ is a covering with $\fz(1)=\infty$. We know that $M\circ\fz$ is a positive harmonic function in $\D$ and by the principle of correspondence of harmonic measure it has zero radial limits almost everywhere. In fact, it can be easily proved that $\fz(re^{i\theta})$ converges to $E$ for almost every $\theta$, hence $M\circ \fz(re^{i\theta})$ tends to zero for almost every $\theta$. 

The Hrglotz representation of $M\circ\fz$ has thus the form
\begin{equation}
\label{Mz}
M\circ\fz (\zeta) = \int \Re\frac{e^{i\theta} + \zeta}{e^{i\theta} - \zeta} d\sigma_s(\theta)\,,
\end{equation}
where $\sigma_s$ is a singular measure on $\T$.

Notice that for any positive harmonic function $m(z)$ with representation like that given by a non-negative measure $\mu$ on $\T$, one can find the point mass of $\mu$ at, say, point $1\in\T$ as follows:
\begin{equation}
\label{mass}
\mu\{1\} = \lim_{r\rightarrow 1} \frac{1-r}{1+r} m(re^{i\theta})\,.
\end{equation}

Let now $F(\Gamma)$ denote the standard fundamental domain of the covering $\fz$, namely, this is 
$$
F(\Gamma)=\{ \zeta\in \D: |\gamma'(\zeta)|<1,\, \forall \gamma \in \Gamma\}.
$$
Consider the branch $\fz^{-1}_{branch}$ of $\fz^{-1}$ from $\C_+$ to $\{\zeta\in F(\Gamma): \Im \zeta>0\}$.

We introduce the new function
\begin{equation}
\label{u00}
w(z):= v(z) +i u(z):= \frac{1+\fz^{-1}_{branch}}{1-\fz^{-1}_{branch}}\,.
\end{equation}
Notice that
\begin{equation}
\label{u0} 
u(x)=0, \, x\in E\,.
\end{equation}
\begin{lem}
\label{easy2}
Let $\sigma_s$ be the singular measure from \eqref{Mz}. Then
$$
\lim_{x\to +\infty} \frac{M(x)}{u(x)} =\sigma_s\{1\} <\infty.
$$
In particular, $M\circ \fz$ has point mass Herglotz representation if and only if
\begin{equation}
\label{Mu}
\lim_{x\to +\infty} \frac{M(x)}{u(x)}>0.
\end{equation}
\end{lem}

\begin{proof}
This is just a change of variable in Lemma \ref{easy1}.


\end{proof}

Now we will give a ``criterion" for \eqref{Mu} to happen. This criterion will be used repeatedly in this work.

Let $E\subset \R_-$, $0,\infty\in E$. The set of complementary intervals (lacunes) of $E$ in $\R_-$ will be called $\cL$.
Let $G(z, t), z, t\in \Omega$ denote Green's function of $\Omega$. Notice a simple thing that $dv(t):=\frac{\partial v}{\partial t} dt$ is a positive measure on each complementary interval $\ell$ of $\R_-\setminus E$, see \eqref{u00}.

Let us fix $P\in (0,\infty)$ and normalize 
$$
M(P)=1, u(P)=1.
$$
We denote
\begin{equation}
\label{G}
\cE:=\int_{\R_-\setminus E} G(P, t) dv(t) >0\,.
\end{equation}
We want to prove the following formula if $\cE<\infty$, then there exists $\rho>0$ such that
\begin{equation}
\label{Riesz}
\rho M(z)  = u(z)+ 2\int_{\R_-\setminus E} G(z,t)dv(t)\,.
\end{equation}

 We enumerate $\ell_1, \ell_2,\dots$ the intervals of $\cL$ and let $E_n$  have only $\ell_1,\dots, \ell_n$ as complementary intervals.

\begin{thm}
\label{int}
Let $E\subset \R_-$, $0, \infty\in E$. Let $u, M$ be functions introduced above and normalized at point $P$. Then
$$
\rho:=\lim_{x\to +\infty} \frac{u(x)}{M(x)} <\infty
$$ 
if and only if
$$
\int_{\R_-\setminus E} G(P, t) dv(t) <\infty\,.
$$
In this case, for any $z\in \Omega$ one has $\int_{\R_-\setminus E} G(P, t) dv(t) <\infty$, and the representation \eqref{Riesz} holds.
\end{thm}

\begin{proof}
Let $
\int_{\R_-\setminus E} G(P, t) dv(t) <\infty\,.
$  Put
$$
\rho := 1 + 2 \int_{\R_-\setminus E} G(P, t) dv(t),\, \rho N(z):= u(z) +2 \int_{\R_-\setminus E} G(z, t) dv(t)\,.
$$ 
Then $\rho\Delta N = \Delta u - dv =2(\frac{\partial u}{\partial n} - \frac{\partial v}{\partial t})dt=0$ by Cauchy--Riemann relationships. Thus $N$ is a positive and harmonic function in $\Omega$. Also we assumed that all finite points of $E=\partial \Omega$ are Dirichlet regular. Hence,
$N(x)=0$ for every such $x\in E$. Therefore, $N$ is a Martin function. And by the choice of $\rho$ we have $N(P)=1$. By the uniqueness $N=M$.

\medskip

Now we assume that 
$$
\rho:=\lim_{x\to +\infty}\frac{u(x)}{M(x)} <\infty\,.
$$
We need to prove that integral in \eqref{G} is finite. This is slightly more difficult and let us explain the difficulty. Let $\Omega_n$ be the domain whose boundary is $E_n= \R_-\setminus \cup_{j=1}^n\ell_j$. It is an approximation to $\Omega$ and it is trivial to see that for a normalized Martin function $M_n$ 
$$
\limsup_{x\to +\infty} \frac{M(x)}{M_n(x)} \le 1\,,
$$
so $1/M_n$ is kind of smaller than $1/M$ at infinity.
Let $v_n+iu_n$ be a branch of $\frac{1+\fz_n^{-1}}{1-\fz_n^{-1}}$, where $\fz_n$ is a universal covering of $\Omega_n$ by $\D$, $\fz_n^{-1}(P)=0$.

If $u_n$ were a kind of smaller than $u$ at infinity, then we would get 
\begin{equation}
\label{rn}
\rho_n:=\lim_{x\to +\infty} \frac{u_n(x)}{M_n(x)} \le \lim_{x\to +\infty} \frac{u(x)}{M(x)} =\rho<\infty.
\end{equation}
This kind of estimate, for example, works in the book of Koosis \cite{Ko} by a simple reason that in his case all functions $u_n, u$ are the same subharmonic function $|\Im z|$. Our case is slightly more complicated as we are going to see.

If \eqref{rn}  were true, it becomes easy to finish the proof. In fact, we repeat the first part of our proof but for $\Omega$ replaced by $\Omega_n$.
 Then
$$
\rho_n M_n(z) = u_n (z) + \int_{\R_-} G_n(z,t) dv_n(t).
$$
Setting $z=P$ we get
$$
\int_{\R_-\setminus E_n} G_n(z,t) dv_n(t)\le \rho_n \le \rho<\infty
$$
by \eqref{rn}.

Then by Fatou's lemma
$$
\int_{\R_-\setminus E} G(z,t) dv(t)\le  \rho<\infty
$$
and consequently, by the first part of the proof
\begin{equation}
\label{fla}
\rho M(z) = u(z) + 2\int_{\R_-\setminus E} G(z, t) dv(t),\,\,\forall z\in \Omega.
\end{equation}
From \eqref{fla} one concludes that
\begin{equation}
\label{rho1}
1+2\int_{\R_-\setminus E} G(P, t) dv(t) =\rho=\lim_{x\to +\infty}\frac{u(x)}{M(x)}\,.
\end{equation}
Unfortunately this reasoning has a serious flaw, namely $u_n$ is actually bigger than $u$, not smaller as we needed.

Still, the crucial inequality \eqref{rn} is correct, but its proof is slightly more sophisticated. Let us prove it now.
Let $\fz_n$  be a universal covering of $\Omega_n$ by $\D$, denote also  by $\tau$ the covering of $\Omega$ by $\Omega_n$ such that $\fz= \tau\circ\fz_n$. Denote by $m$ positive harmonic in $\Omega_n$ function $M\circ \tau$, then $m(P)=1$. It is easy to see that
\begin{equation}
\label{mnm}
\lim_{x\to +\infty}\frac{m(x)}{M_n(x)}\le 1.
\end{equation}
On the other hand $u_n= u\circ \tau$, $m= M\circ \tau$, therefore,
\begin{equation}
\label{unm}
\lim_{x\to +\infty}\frac{u_n(x)}{m(x)}= \lim_{x\to +\infty}\frac{u(x)}{M(x)}\le \rho<\infty.
\end{equation}

Combining \eqref{mnm} and \eqref{unm} we get \eqref{rn}. As we saw this finishes the proof of \eqref{G}, \eqref{fla} and \eqref{rho1}.

\end{proof}

\section{An example of purely continuous lifting of a point mass}
\label{purecont}

Now we are going to build the domain $\Omega$ for which \eqref{Mu} fails. In particular, the positive harmonic function $M\circ \fz$ will have a continuous singular measure, where $M$ is the Martin function at infinity of $\Omega$. We already proved that
condition \eqref{Mu} holds if and only if 
\begin{equation}
\label{intGv}
\int_{\R_-\setminus E} G(P, t) dv(t)  <\infty
\end{equation}
is satisfied for some (and then for all $P\in \Omega$). 
This follows from \eqref{rho1}.

Let us now consider $\R_-\setminus E:= \bigcup_{n\in \mathbb{Z}} \lambda^n (a_1, b_1)$, where $\lambda >1$ and assume that \eqref{intGv} holds for this set.

For the given set $G(\lambda z, \lambda t) = G(z, t)$, and the conformal map $w$ onto the fundamental domain and Martin's function satisfy
$$
w(\lambda^n z) = \rho_1^n w(z),\,\,\text{for some}\,\, \rho_1>1, n\ge 1,
$$
$$
M(\lambda^n z) = \rho_2^n M(z),\,\,\text{for some}\,\, \rho_2>1, n\ge 1.
$$
On the other hand,
\begin{equation}
\label{MuSmall}
\limsup_{x\to +\infty} \frac{M(x)}{u(x)} \le \sigma_s(\{1\})<\infty,
\end{equation}
where $\sigma_s$ is the measure of singular inner function $M\circ \fz$. See Lemma \ref{easy2}.

Therefore, we immediately get
$$
\rho_2\le \rho_1.
$$
Moreover, inequality \eqref{Mu} provides of course the opposite inequality, and, thus,
\begin{equation}
\label{rhorho}
\rho_1=\rho_2=:\rho.
\end{equation}

We are going to bring this to a contradiction. We recall the reader that $dv\ge 0$ in \eqref{intGv}.
Now use $G(\lambda z, \lambda t) = G(z, t)$ and $v(\lambda z)= \rho v(z)$ to rewrite \eqref{intGv} (which is equivalent to \eqref{Mu} as we saw), as follows (positivity and then Fubini's theorem)
\begin{equation}
\label{sum1}
\int_{a_1}^{b_1}\sum_{n\in\mathbb{Z}} \rho^n G(\lambda^{-n} z, t) dv(t) <\infty.
\end{equation}

\begin{lem}
\label{MGseries}
For every $\lambda>1$ the series
$$
\sum_{k\ge 0} M(\lambda^k) G(\lambda^{-k}, t) =\infty
$$
for all $t\in (a_1, b_1)$.
\end{lem}

\begin{proof}
From the definition of $M$ we know that
$$
M(\lambda^k) = \lim_{n\to\infty}\frac{G(\lambda^k, \lambda^{n+k}t)}{G(1, \lambda^{n+k}t)}, t\in (a_1, b_1).
$$
The general term of the series in question is, thus, precisely
$$
\lim_{n\to\infty}G(1, \lambda^kt)\frac{G(1, \lambda^{n}t)}{G(1, \lambda^{n+k}t)}.
$$
Let us assume for a moment that we can prove

\begin{equation}
\label{geom}
\frac{G(1, \lambda^{n+k}t)}{G(1, \lambda^{n}t)G(1, \lambda^{k}t)}\le C(t)<\infty, \,\,t\in (a_1, b_1).
\end{equation}
Then the general term above is strictly positive, hence lemma is proved. We are left to check \eqref{geom}.
This is the same as to prove
\begin{equation}
\label{geom1}
G(\lambda^{-(n+k)}, t)\le CG (\lambda^{-n}, t)G(\lambda^{-k}, t), \,\,t\in (a_1, b_1).
\end{equation}
To prove \eqref{geom1} we will use the following estimate of the harmonic measure $\omega:=\omega_\Omega$ of our domain $\Omega$, see the definition in \cite{Ne} for example.
\begin{equation}
\label{Gom}
G(x,t) \asymp \omega([-x,0],1), \, x>0,
\end{equation}
where the constants of comparison depend only on $t$ and are positive and finite. Claim \eqref{Gom} is well-known, and its varinats  can be found  in \cite{JK}, \cite{V} for example.

Accordingly, we are left to prove that
\begin{equation}
\label{geom2}
\omega (E_{n+k}, 1) \le C_1 \omega (E_n,1) \omega (E_k,1),
\end{equation}
where $E_m:= E\cap [-\lambda^m, 0]$. To do that, let us consider the disc $D_k$ centered at  $0$  and passing through the point $\lambda^{-k}c$, where $c=\frac12(a_1+b_1)$. Let $T_k:=\partial D_k$. Then 
$$
\omega(E_{n+k}, 1) = \int_{\partial (\Omega\setminus D_k)} \omega (E_{n+k}, z) d \omega_{\Omega\setminus D_k} (z, 1).
$$
By Harnack's inequality,  with constant $C$ independent of $n, k$, we will then have
\begin{equation}
\label{Harn1}
\omega(E_{n+k}, 1)\le C\, \omega (E_{n+k}, \lambda^{-k}c) \omega_{\Omega\setminus D_k} (T_k, 1).
\end{equation}
Now let us notice several things: 

1) by the self-similarity of $\Omega$ we immediately have
$$
\omega (E_{n+k}, \lambda^{-k}c) = \omega (E_{n}, c);
$$

2) by Harnack's principle again  $\omega (E_{n}, c)  \le \omega (E_{n}, 1) $;

 3) the following inequality holds
\begin{equation}
\label{vspom}
\omega_{\Omega\setminus D_k} (T_k, 1)\le C\omega (E_k,1).
\end{equation}
It  follows from the  fact: 
$\omega (E_k, z) \ge c_0>0$ uniformly for $z\in \T_k$, and the constant $c_0$ does not depend on $k$. This latter claim is the combination of self-similarity
(everything is like for $k=0$) and Harnack's principle. Now \eqref{vspom} follows by comparing two harmonic functions $\omega (E_k, z)$ and  $\omega_{\Omega\setminus D_k} (T_k, z)$ in the same domain $\Omega\setminus D_k$. At any given point $z$ of  the boundary of this domain they are both either vanish, or
$\omega (E_k, z) \ge c_0 \cdot \omega_{\Omega\setminus D_k} (T_k, z)$. Therefore, $\omega (E_k, z) \ge c_0 \cdot \omega_{\Omega\setminus D_k} (T_k, z)$ holds everywhere on the boundary, and hence, everywhere inside the domain $\Omega\setminus D_k$, in particular at point $z=1$. This is exactly \eqref{vspom} with $C=c_0^{-1}$. 

Combining 1), 2), 3), and \eqref{Harn1} we obtain \eqref{geom2}, and
correspondingly, \eqref{geom1}  is proved, and Lemma is proved.
\end{proof}

We assumed the valididty of \eqref{Mu}, and we obtained that then $M(\lambda^k)= \rho^k M$ with some positive constant $M$. Plugging this into the series in Lemma \ref{MGseries} we obtain that for all $t\in (a_1, b_1)$ the series $\sum_{k\ge 0} \rho^k G(\lambda^{-k},t) =\infty$. But \eqref{sum1} says that such series converges  for almost every $t\in (a_1, b_1)$. This is a contradiction. Therefore, \eqref{Mu} does not hold for this example. We just proved that the lifting of the point mass can become a purely continuous singular measure.

\section{Lifting of a point mass to a pure point spectrum from the infinitely connected domain $\Omega$}
\label{PurePoint}

Point mass lifts to a pure point measure if $\R_-\setminus E$ consists of finitely many open intervals  (lacunes or gaps). However, we want to construct the example, when the map ($M$ is the Martin function of
$$
\Theta:= -\wt M+ iM
$$
has unbounded ``teeth". So the number of lacunes should be infinite and they should substantially grow in length (we will see the details). This map is a conformal map of $\C_+$ onto a comb domain  $\Pi:=\C_+\setminus \cT$, where $\cT$ is $\cup_{k=0}^\infty T_k$, and $T_k$ is a vertical segment attached to a point $\omega_k\in \R$. By our normalization all $\omega_k<0$, $k\ge 1$, and we have an infinite tooth $T_0=\{iY: Y\ge 0\}$. In fact, $\Pi$ lies in the quarter plane $Q:=\{X<0, Y>0\}$. Let us write each tooth as follows: $T_k= \{\omega_k+it, 0\le t\le h_k\}$, where $h_k=M(m_k)$ and $m_k$ is the maximum point of the function $M$ on the $k$-th lacune (=complimentary interval of $E$) $(a_k, b_k)$. The conformality of $\Theta$ is a well-known fact and can be found, for example, in \cite{EYu}.

It is rather easy to see that if the hight of teeth $h_k$ satisfy
\begin{equation}
\label{bound}
\sup_k h_k = \sup_k M(m_k) <\infty,
\end{equation}
then 
$$
M(iy) \ge c\sqrt{y}, \,\, y\to +\infty,\,\, c>0.
$$
Notice that function $M$ always satisfies the opposite inequality, and therefore $u$ satisfies the opposite inequality, see \eqref{Riesz} for example. Then \eqref{Mu} holds, and $M\circ \fz$ is the Poisson integral of pure point measure, as was proved in Theorem \ref{int}.

\bigskip

However, we want the supremum in \eqref{bound} to be infinite, and \eqref{Mu} to hold at the same time.

\smallskip

\noindent{\bf Remark.}
In fact, by constructing such an example (the supremum in \eqref{bound} is infinite, but the lifting is pure point) we answer a question of Chris Bishop and Michael Sodin. 
They noted that if $\sup_{k\ge 1} h_k=\infty$, then one can find on the universal covering  uncountably many different rays such that $M\circ\fz$ goes to infinity along each of them. That is, it looks like the corresponding measure is ``supported" on an uncountable set of points and ``therefore" it is  not pure point. In this sense our answer is slightly counterintuitive.  However, the lifting of a point measure can be pure point measure evein if $\sup_K h_k=\infty$. It other words, \eqref{bound} is sufficient, but not necessary for such a lifting.

\bigskip

The example, where supremum in \eqref{bound} is infinite, and \eqref{Mu}  holds, can be achieved in many ways. One of them is only sketched here, another is given in details. Let us consider $\wt E:= \{x\in \R: -x^2\in E\}$.  One can easily see that formula
$$
\mathcal{M}(z) := M(-z^2),\,\, z\in \C_+,\,\, \mathcal{M}(z) :=\mathcal M(\bar z), \,\, z\in \C_-
$$
defines Martin function at infinity for 
$\C\setminus \wt E$. The question when 
\begin{equation}\label{mgm}
\mathcal{M}(z)\approx cz, \quad c>0,
\end{equation}
when $z$ approaches infinity has been thoroughly studied by Benedicks \cite{Be}. One can reconcile his condition (this requires some work) with $\sup_k M(m_k)=\infty$.

\bigskip

We propose another way to build an example, where teeth may grow unboundedly, but still \eqref{Mu} holds with a not necessarally  maximal possible rate of grow for $z$  as in \eqref{mgm}.
That is,  we will be reconciling \eqref{bound} and \eqref{intGv}:
$$
\int_{\R_-\setminus E} G(P, t) dv(t)  <\infty.
$$

We will use the ``sliding hump method" to build our domain $\Omega$, function $M$, and comb domain $\Pi$ inductively. So we build the sequence of domains $\Omega_k=\C\setminus E_k$, where 
$$
\R_-\setminus E_k = \bigcup_{i=1}^k (a_i^k, b_i^k)=: \bigcup_{i=1}^k L_i^k.
$$
We will make the lacunes $L_k^k:=(a_k^k, b_k^k)$ to tend to $-\infty$ very fast.

\medskip

Let the sequence of ``teeth height" $\{h_m\}_{m=1}^\infty$, such that $\lim_m h_m=\infty$, be fixed.

\medskip

Let $k$-th generation is built. Functions $\Theta_k=-\wt M_k +i M_k$ and $w_k= v_k+i u_k$ are conformal maps of $\C_+$. The first one is  onto the comb domain $\Pi_k:=Q\setminus \cT_k$,
$\cT_k:= \cup_{i=1}^k T_k$, 
$$
T_j=\{\omega_j + iY, 0\le Y\le h_j\}, \, j=1, \dots, k.
$$
The image of the second map is
$O_k=Q\setminus \cup_{j=1}^k D_j^k$, where $D_j^k$, $j=1,\dots, k$ are closed disjoint discs with centers $c_j^k$, $c_k^k<\dots< c_1^k<0$.

We put $\Pi_{k+1}= \Pi_k \setminus T_{k+1}$, where 
$$
T_{k+1}=\{\omega_{k+1} + iY, 0\le Y\le h_{k+1}\}.
$$
The point $\omega_{k+1}$ of attachment of this tooth will be chosen very negative. If $|\omega_{k+1}|$ is chosen very large, we can ensure that with any fixed positive $\e$ we have
$|a_{j}^{k+1} - a_j^k|<\e$, $|b_{j}^{k+1} - b_j^k|<\e$, $j=1, \dots, k$. 
$$
|G_{k+1}(0, t) - G_k(0, t)| <\e,
$$
and
$$
\bigg|\frac{d v_{k+1}}{dt} - \frac{d v_{k}}{dt}\bigg|<\e,
$$
uniformly in $t$, if $t$ runs over a fixed compact set $S_k$.

In particular, the integrals over all the $(k+1)$-th lacunes except the last one will be easily controlled:
\begin{equation}
\label{indInt}
|\sum_{j=1}^k \int_{L_j^{k+1}} G_{k+1}(0, t) dv_{k+1}(t) - \int_{\R_-\setminus E_k} G_k(0,t) dv_k(t)|<\e.
\end{equation}

Let us write formula \eqref{fla} in $\Omega_k$ and in $\Omega_{k+1}$ (we use also that $u_m, M_m$ are normalized at point $0$). Then we have
\begin{equation}
\label{Omk}
1= u_k(0) = \rho_k -\frac1{\pi} \int_{\R\setminus E_k} G_k(0, t) dv_k(t),
\end{equation}

\begin{equation}
\label{Omk1}
1=  \rho_{k+1} -\frac1{\pi} \int_{\R\setminus (E_{k+1}\setminus L_{k+1}^{k+1})} G_{k+1}(0, t) dv_{k+1}(t) -\frac1{\pi} \int_{L_{k+1}^{k+1}} G_{k+1}(0, t) dv_{k+1}(t) ,
\end{equation}

If we combine \eqref{indInt} with \eqref{Omk}, \eqref{Omk1}, we immediately get
\begin{equation}
\label{rhok1}
|\rho_{k+1} -\rho_k - \frac1{\pi} \int_{L_{k+1}^{k+1}} G_{k+1}(0, t) dv_{k+1}(t)| \le 2^{-k-1},
\end{equation}
if above we choose $\e$  accordingly.

\bigskip

Now we are going to estimate $\int_{L_{k+1}^{k+1}} G_{k+1}(0, t) dv_{k+1}(t) $. To do that we notice that
$w_{k+1}$  maps interval $L_{k+1}^{k+1}$ onto the half circle $\partial D_{k+1}^{k+1}$ in a one-to-one monotone fashion. In particular, the increase of $v_{k+1}$ on this interval is exactly twice the radius of this circle, which is $2u_{k+1}(c_{k+1})$.
Hence,
$$
\int_{L_{k+1}^{k+1}} dv_{k+1}(t) \le 2 u_{k+1}(c_{k+1}) \le 2\rho_{k+1}  M_{k+1}(c_{k+1}),
$$
where the last inequality follows just from \eqref{fla}. We got then
\begin{equation}
\label{vk1}
\int_{L_{k+1}^{k+1}} dv_{k+1}(t) \le 2\rho_{k+1}h_{k+1}.
\end{equation}

To estimate $\int_{L_{k+1}^{k+1}} G_{k+1}(0, t) dv_{k+1}(t) $ we are now left to estimate $G_{k+1}(0, t)$ uniformly, when $t\in L_{k+1}^{k+1}$.

Let us fix
$$
\e_{k}:= 2^{-{(k+1)}} h_{k+1},
$$
but let us move $\omega_{k+1}$ close to $-\infty$.

Recall that $L_{k+1}^{k+1} = (a_{k+1}^{k+1}, b_{k+1}^{k+1})$. We skip the indices and write $\omega:=\omega_{k+1}$, $L= (a, b)$. Of course, $a, b\to-\infty$ when $\omega\to -\infty$. Let us prove that eventually $|a|\le 2|b|$. Inessentially changing our normalization we can think that the conformal map $\Theta:=\Theta_{k+1}$ maps $A_0= i$ to $B_0=-1+i\in Q$. Map $\Theta$ maps $\C_+$ onto $\Pi:= \Pi_{k+1}$. Consider the new conformal map $\theta :=1/\Theta(1/z)$ mapping $a_0:=1/A_0$  into $b_0=1/B_0$. Let $\alpha$ be harmonic measure of $\C_+$ with respect to $a_0$, and $\beta$ be harmonic measure of $\mathcal{S}:= 1/\Pi$ (the image of $\Pi$ under the map $1/z$) with respect to $b_0$.

Suppose $|a|>2|b|$ when $\omega$ goes to $-\infty$ over a certain sequence. Then it is immediate that
$$
\alpha((1/b, 1/a))/\alpha((1/a, 0)) \ge c>0
$$
for all such large $\omega$, where $c$ does not depend of $\omega$ and is, in fact, an absolute constant. This is just because in $\C_+$ the harmonic measure of a small interval evaluated at $b_0$ is comparable with the lenth of the interval if it lies not far from $0$.  

By the principle of preservation of harmonic measure we then have
$$
\beta(1/T_{k+1})/\beta((1/\omega_{k+1}, 0)) \ge c>0,
$$
where $1/T_{k+1}$ is the image of $T_{k+1} =\{\omega_{k+1} +it, 0\le t\le h_{k+1}\}$ under the map $1/z$.
But it is easy to see that $\beta(1/T_{k+1} )\le C \frac{\sqrt{h_{k+1}}}{|\omega_{k+1}|}$ with absolute $C$,  for all large $\omega_{k+1}$. On the other hand, it is also easy to see that
$\beta((1/\omega_{k+1}, 0)) \ge c_k\frac{1}{\sqrt{|\omega_{k+1}|}}$ with positive $c_k$, which depends on $k$ but does not depend on $\omega_{k+1}$.
Then 
$$
0<c\le\beta(1/T_{k+1})/\beta((1/\omega_{k+1}, 0)) \le \frac{C}{c_k}\sqrt{\frac{h_{k+1}}{|\omega_{k+1}|}}.
$$
Tending $\omega_{k+1}$ to $-\infty$ and keeping $h_{k+1}$ fixed we come to contradiction.

\bigskip

Thus, we have just proved that for all large $\omega_{k+1}$
$$
|a_{k+1}^{k+1}|\le 2|b_{k+1}^{k+1}|.
$$

If we make the fractional linear transformation $m$
sending $0$ to infinity, infinity to $0$ and sending $b_{k+1}^{k+1}$ into point $1$, it will send $\Omega_{k+1}$ to a domain we will call $\mathcal{D}_{k+1}$. We can see that we need to estimate $G_{\mathcal{D}_{k+1}}(s, \infty)$ uniformly when $s\in I_{k+1}:= (m(a_{k+1}^{k+1}),1)$. 
We just proved 
$$
m(a_{k+1}^{k+1})\ge 1/2,\, |I_{k+1}|\le 1/2.
$$

Hence domain $\mathcal{D}_{k+1}$ has boundary, which definitely includes segments $[0,1/2]$ and $J_{k+1}:=[m(b_{k+1}^{k+1}), m(a_{k}^{k+1})]=[1, m(a_{k}^{k+1})]$. The latter segment has the lenght of the order $b_{k+1}^{k+1}/a_k^{k+1}$, which is as large as we wish. 

We saw that lacune $I_{k+1}=(m(a_{k+1}^{k+1}), 1)$ is contained in $(1/2, 1)$.
 By the principle of majorization of Green's function (bigger domain, bigger function, if measured at the same points) we can see now that
$$
\max_{s\in I_{k+1}} G_{\mathcal{D}_{k+1}}(s, \infty) \to 0, \,\, \text{when}\,\,\omega_{k+1}\to \infty, \,\, \text{if}\,\, |h_{k+1}| \,\,\text{stays fixed}.
$$
In fact, we just compare $G_{\mathcal{D}_{k+1}}(s, \infty)$, $s\in I_{k+1}$, with $G_{J_{k+1}}(s, \infty)$, where the latter is Green's function of $\C\setminus ([0,1/2]\cup  J_{k+1})$. The uniform smallness of $G_{J_{k+1}}(s, \infty)$ on $(1/2, 1)$  is obvious becaise $J_{k+1}$ becomes as long as we wish.

In particular, for large $\omega_{k+1}$ we have
$$
\max_{t\in L_{k+1}^{k+1}} G_{k+1}(0, t) \le \max_{s\in I_{k+1}} G_{\mathcal{D}_{k+1}}(s, \infty) \le 2^{-(k+1)} h_{k+1}^{-1}.
$$

Combine this with \eqref{vk1} to get 
$$
\int_{L_{k+1}^{k+1}} G_{k+1}(0, t) dv_{k+1}(t)  \le 2^{-k}\rho_{k+1}.
$$
This and \eqref{rhok1} give us
$$
\rho_{k+1}\le \rho_k + C 2^{-k}
$$ with absolute constant $C$. Then we get that $\rho_m$ will be bounded and this implies \eqref{intGv}.

We are done. We constructed the example with infinitly growing teeth, but with \eqref{intGv} and \eqref{Mu} satisfied, which means that the point mass will be lifted to pure point measure inspite of unboundedness of the teeth. 

\section{Martin function dominates outer functions. The proof of Theorem \ref{main1}.}
\label{Mouter}

A simple corollary of formula \eqref{mass} is 
that for any outer  function $F_{out}(z)$ in the unit disc $\D$ one can claim that
\begin{equation}
\label{outer1}
\lim_{r\rightarrow 1} \frac{1-r}{1+r}\log|F_{out}|(r) = 0.
\end{equation}

Notice that Martin's function for the disc $\D$ at point $1$ is exactly $M(z)=\frac{1+z}{1-z}$. Therefore, \eqref{outer1} means
\begin{equation}
\label{outer2}
\lim_{r\rightarrow 1} \frac{\log|F_{out}|(r)}{M(r)} = 0.
\end{equation}

It is natural to ask the same question in more complicated domains. At least in the case of our Denjoy type domain $\Omega= \C\setminus  E, E\subset \R_-$, we will prove the validity of this dominance. This is exactly the claim of Theorem \ref{main1}.

\medskip

\noindent
{\bf Remark.} Notice that this dominance in Denjoy domain does not follow from lifting  to $\D$ and the dominance \eqref{outer2} in $\D$. We already explained, that the lifting of $M$ in $\Omega$ may have nothing to do with Martin function in $\D$, because the Herglotz measure the lifting of $M$ in $\Omega$ to $\D$ can be even  continuous.

\begin{lem}
\label{omgrows}
Let $E$ be a compact set in $\R_-=(-\infty, 0)$. Let each point of $E$ be Dirichlet regular. Consider $O:=\C\setminus E$, and let $\omega(E', z)$ denote its harmonic measure, $E'\subset E$. Let $\omega_a(z):= \omega ((-\infty, a]\cap E, z)$, for $a\in \R_-$, where $a$ is a point of accumulation of $E$ from the left. Then for any such $a$ the function $x\to \omega_a(x)$ is non-decreasing on $\R_+$.
\end{lem}

\begin{proof}
Let $0<x<y$, $a<0$, and let $\omega_a(y)<\omega_a(x)$.  We want to come to a contradiction. Let $\omega_a(y)<t<\omega_a(x)$. Consider the set $\{z\in \C\cup\infty: \omega_a(z) <t\}$. It is an open set in $\C\cup\infty$ containing $y$. It may or may not contain $\infty$.  What we care about is whether its connected component $\mathcal{C}$ containing $y$ is such that  it contains any point of $E$. Supoose it does not. Clearly $\omega_a(z)$ is a harmonic function in $\mathcal{C}$ and it is continuous in $\bar{\mathcal{C}}$. It is identically $t$ on the boundary of $\mathcal{C}$ (with the possible exception of point $a$, if $a$ happens to be in its boundary), hence inside $\mathcal{C}$ it is also $t$. This is a contradiction with $\omega_a(y)<t$.

Consequently, $\mathcal{C}$ contains some point $b$ of $E$.
Let us show that $b\ge a$. Suppose $b<a$. Then
$$
\omega_a(z)=1, \, \text{for all}\,\,z\in E\,\,\text{in a small neighbourhood of}\,\, b,
$$
and this means that inside the domain $\mathcal{C}$ there are points where value of $\omega_a$ is as close to $1$ as we wish. But these values (by the definition of $\mathcal{C}$) should be  strictly smaller than $t<1$. This is a contradiction.

Let us show that $b\neq a$. Suppose $b=a$.  We use that $a$ is a point of accumulation of $E$ from the left. Then again  inside the domain $\mathcal{C}$ there are points where value of $\omega_a$ is as close to $1$ as we wish, and this is a contradiction.

Points $b, y$ lie in the open connected set $\mathcal{C}$, so they can be connected by a  piecewise linear arc $\gamma$ inside $\mathcal{C}$. We can assume without the loss of generality that each straight segment of $\gamma$ is not parallel to $\R$. 

On $\gamma$ our function $\omega_a(z)$ is strictly smaller than $t$. By symmetry the same holds on $\bar{\gamma}$. Lets us put $\Gamma:=\gamma\cup\bar{\gamma}$ and consider  the component of infinity of $\C\setminus \Gamma$, let it be $A$, and let the  bounded components be $B_1,\dots, B_n$.  Notice three things: 
\begin{itemize}
\item[1)] as $b\in E$, then $b<0$,
\item [2)] $[b, y] \subset \R\cap\bigcup_{i=1}^n \bar{B_i}$, 
\item [3)] on $\bar{B_i}$ function $\omega_a(z)$ is strictly smaller thant $t$. In fact, this is true on the boundary of each domain $B_i$. Every point of $\partial B_i$  is in $\Gamma\cup (E\cap [b,0])$, and $b>a$,  so $\omega_a(z)$ is either less than $t$ or vanishes at every boundary point, and then maximal principle works.
\end{itemize}

Combining these properties we see that  
$$
\omega_a(z)<t\,\, \forall\,\, z\in [b, y].
$$
But $b<0<x<y$, so $\omega_a(x)<t$. This contradicts the choice of $t$. We are done.
\end{proof}

In what follows we always assume that all points of the boundary are Dirichlet regular.

\smallskip
\noindent
{\bf Remark.} The fact that $a$ is a point of accumulation of $E$ from the left was used in the proof, but it is not essential for the statement. In fact, if $a$ has no points of $E$ in $(a-\e, a)$, one can just change $a$ without changing function $\omega_a$.

\begin{lem}
\label{twopo}
Let $E\subset (-\infty, 0)$. Let $\Omega=\C\setminus E$.  Let $b<a, b, a\in E$. Then
$$
\frac{\omega_a(x)}{\omega_a(0)} \le \frac{\omega_b(x)}{\omega_b(0)},\, \forall x >0.
$$
\end{lem}

\begin{proof}
We can first prove this for compact $E$ and then this follows for  $E$ containing $-\infty$ by an obvious limiting procedure.

We assume first that $b$ is an accumulation point of $E$ on the left. We saw that this assumption is without loss of generality.
 Let $A=1/\omega_a(0), B=1/\omega_b(0)$. Consider function 
$$
u(z):=B\omega_b(z) - A\omega_a(z)
$$
and suppose that for some $x>0$ we have $u(x)<-\e<0$. Let $O$ be a connected component of the set, where $u<-\e$, containing point $x$. The open set $O$ must conatin some point of $E$, otherwise we come to contradiction as in the previous lemma.

Suppose $c\in E$ and $c<b$. Notice that $B\ge A$ and $\omega_b(c)=\omega_a(c)=1$. Then $u(c)\ge 0$. This is impossible. By the same reason $c\neq b$. Otherwise, using that $c=b$ is an accumulation point of $E$ on the left, we get some points inside $O$, where $u\ge 0$.

Hence, $c>b$. We again connect $c$ and $x$ by a piecewise linear path $\gamma$ inside the domain $O$.
We may assume that all straight segments of $\gamma$ are not parallel to $\R$.
Again we consider $\Gamma=\gamma\cup\bar{\gamma}$ on which $u<-\e$ by construction. 
Let $\{B_i\}_{i=1}^n$ be bounded connected components of $\C\setminus \Gamma$. 
Let us look at $\partial B_i$. Its points belong either to $\Gamma$ (where $u<-\e$), or to $E\cap (b, 0)$. But at such points $\omega_b(z)=0$. Therefore, at such points $u(z) \le 0$. 

We conclude that $u<0$ at any point of $\bar{B_i}\setminus E$, $i=1,\dots, n$. 

As in the previous lemma, we obviously see that $[c,x] \subset \R\cap \bigcup_{i=1}^n\bar{B_i}$. Hence, for any $z\in  [c,x]$ which is not in $E$ we conclude $u(z)<0$. But point $0$ is exactlly in $[c,x]\setminus E$. So $u(0)<0$. But by definition $u(0)=0$. We came to contradiction and proved the lemma.

\end{proof}

\begin{thm}
\label{Momega}
Let $E\subset (-\infty, 0)$. Let $\Omega=\C\setminus E$. We assume that $\infty\in \partial \Omega$. Let $M$ be its Martin function at infinity normalized as follows: $M(0)=1$. Let $a\in E$ and $\omega_a$ be harmonic measure of $E\cap(-\infty, a]$ with respect to $\Omega$. Then
$$
\frac{\omega_a(x)}{\omega_a(0)} \le M(x), \, \forall x>0.
$$
\end{thm}

\begin{proof}
We use $M(z)=\lim_{b\to-\infty}\frac{\omega_b(z)}{\omega_b(0)}$. Then the claim follows fro Lemma \ref{twopo} immediately.

\end{proof}

\begin{lem}
\label{twop1}
Let $E\subset (-\infty, 0)$. Let $\Omega=\C\setminus E$.  Let $b<a, b, a\in E$.  Let $\omega_{ba}(z) :=\omega_a(z)-\omega_b(z)=\omega([b,a], z)$. Then
$$
\frac{\omega_{ba}(x)}{\omega_{ba}(0)} \le \frac{\omega_b(x)}{\omega_b(0)},\, \forall x >0.
$$
 We assume that $\infty\in \partial \Omega$. Let $M$ be  Martin function at infinity normalized as follows: $M(0)=1$. Then
$$
\frac{\omega_{ba}(x)}{\omega_{ba}(0)} \le M(x),\, \forall x >0.
$$
\end{lem}

\begin{proof}
The first inequality is a direct consequence of Lemma \ref{twopo}. The second inequality combines the first one and Theorem \ref{Momega}.

\end{proof}

\begin{thm}
\label{main12}
Let $\Omega=\C\setminus E$. We assume that $\infty\in \partial \Omega$. Let $M$ be its Martin function at infinity. Let $\phi$ be a function summabe with respect to harmonic measure $\omega(., z)$ of $\Omega$. Then 
$$
\int\phi(t) \omega (dt, x) = o(M(x)),\,\, x\to +\infty.
$$
\end{thm}

\begin{proof}
Let $N<0$. Let $\rho(t, x):=\frac{\omega (dt, x)}{\omega(dt,0)}$. Then
$\int\phi(t) \omega (dt, x)  = \int_{-\infty}^N\phi(t) \rho(t,x) \omega(dt,0) +\int_N^0\phi(t) \rho(t,x)\omega (dt, 0) =: I+II$. By Lemma \ref{twop1} we conclude that uniformly $\rho(t, x) \le M(x)$. Hence
$$
I\le M(x)\int_{-\infty}^N \phi(t) \omega(dt, 0) \le \e M(x)
$$
if $N$ is large enough by absolute value because of the fact that $\phi\in L^1(d\omega)$.

We notice that by Lemma \ref{twop1} we can write that  for $t\in [N, 0]\cap E$
$$
\rho(t,x) \le \frac{\omega_N(x)}{\omega_N(0)}\le \frac{1}{\omega_N(0)}=: C_N.
$$
Then $II/ M(x) \le (C_N \int |\phi(t)| \omega(dt, 0))/M(x)$, and it tends to zero when $x$ tends to plus infinity just because $M(x)$ goes to infinity. Theorem is proved.

\end{proof}

Notice that we just proved Theorem \ref{main1}.

\section{Martin function dominates the logarithm of Blaschke product with zeros on $\R_-$}
\label{MB}

We already proved that $h\le a$, see \eqref{ah}. Now we want to show some cases when we can prove the equality $h=a$.

Let $\Omega=\C\setminus E$. Recall that all point of $E$ are assumed to be Dirichlet regular. We assume that $\infty\in \partial \Omega$.  We can assume that
\begin{equation}
\label{infty}
\text{infinity is also a Dirichlet regular point of}\,\, E.
\end{equation}
Actually for us it will be enough to have another condition \eqref{MG} below.
Let $M$ be its Martin function at infinity.  For what follows we would need 
\begin{equation}
\label{MG}
\forall c\in \R_- \forall x\in \R_+, \,\, G(c, x) \le C(c)<\infty.
\end{equation}

\noindent
{\bf Remark.} Notice that some condition of the type \eqref{MG} is needed, because otherwise it may happen
that there is no Martin function at infinity at all, or it may happen that $M(x)\asymp \log x, x\to+\infty,$ and $G(c, x)\asymp \log x, x\to +\infty,$ and then the statement  below is obviously wrong. Assumption \eqref{MG} is satisfied, for example, if $\infty$ is a Dirichlet regular point of the boundary $E$ of the domain $\Omega$.

\begin{thm}
\label{MBth}
Let $\{c_k\}_{k=1}^\infty$ be points on $\R_-$, $c_1>c_2>\dots >c_m>\dots$. Let $\sum_{k=1}^\infty G(c_k, 0) <\infty$. Then 
$$
\sum_{k=1}^\infty G(c_k, x) = o(M(x)), \,\, x\to +\infty,
$$
if assumption \eqref{MG} is satisfied.
\end{thm}

We need the following lemma.

\begin{lem}
\label{2G}
Let $a, b\in \R_-\cap \Omega$, $b<a$. Then
$$
\frac{G(a, x)}{G(a,0)}\le \frac{G(b, x)}{G(b,0)},\,\, \forall x\in \R_+.
$$
\end{lem}

\begin{proof}
Suppose that for $x\in \R_+$ we have an opposite inequality. Denote $B:=\frac{1}{G(b,0)}, A :=\frac{1}{G(a,0)}$ and consider
$$
u(z):= B G(b, z) - A G(a, z).
$$
We assumed that $u(x)<-\e<0$. Consider the connected component 
of the open set $\{z\in \C: u(z) <-\e\}$ that contains point $x$. Call this component $O$. 

Domain $O$ cannot contain $b$ inside. This is clear, because at the vicinity of $b$ function $G(b, z)$ is close to $+\infty$, and so $u(z)$ has the same property. Moreover, $b\notin \bar{O}$ by the same reason.

Hence,  $u$ is subharmonic in $O$. Suppose $a \notin O$, and $E\cap O=\emptyset$. Then $u$ is harmonic in $O$, and its values on $\partial O$ are identically $-\e$. This is obviously true except may be one point of $\partial O$, namely, $a$, if it so happened that $a\in \partial O$. But $a$ cannot belong to the boundary of $O$, otherwise there would be a sequence of distinct points on this boundary, where $u<-\e$, which is impossible. 

Then automatically harmonic in $O$ function $u$ being identically $-\e$ on $\partial O$ will be identically $-\e$ in $O$. This is a contradiction with $x\in O$.

Hence, we have to assume that either $a\in O$ or $E\cap O\neq \emptyset$. Suppose $a\notin O$.  Then $E\cap O\neq \emptyset$.
We can choose then a point $z\in O$ close  to $E\cap O$,  where $G(a, z)$ and $G(b,z)$ are as close to $0$ as we wish (Dirichlet regularity assumption). Then $u(z)\ge -\e/2$. But on the boundary of $O$ our subharmonic function $u$  is $-\e$. Maximum principle for subharmonic functions does not allow this.

There is only one possibility left: $a\in O$. Then we choose a point $r\in \R_-\cap O$ close to $a$ at which we have
$$
u(r)<-\e.
$$
Points $r$ and $x$ are in the same connetced open set $O$. Hence we can connect them by a piecewise linear path $\gamma$ inside $O$, consisting of finitely many straight segments. Each straight segment can be chosen not to be parallel to $\R$.

As before, denote $\Gamma := \gamma\cup\bar{\gamma}$, and notice that $u(z) <-\e, \forall z\in \gamma$. Then, by symmetry,
$$
u(z) <-\e, \forall z\in \Gamma.
$$

Again consider all bounded components of $\C\setminus \Gamma$, let them be $\{B_i\}_{i=1}^n$.

Obviously $[r, x]\subset \bigcup_{i=1}^n \bar{B_i}$. By construction,  $\gamma$ is inside $O$, and, by symmetry, $\Gamma\subset O$. It may happen by chance that $b$ belongs to one of these domains $B_i$, let us say to $B_1$. Then of course $[r, x]\subset \bigcup_{i=2}^n \bar{B_i}$. In each of $B_i$ that does not contain $b$ function $u$ is sabharmonic, it is $<-\e$ on $\partial B_i\cap\Gamma$ and it is $=0$ on $\partial B_i\cap E$. Therefore, each $s\in [r, x]\setminus E$ is either a point of intersection of $\gamma$ and $\R$ or it is in $B_i\setminus E$, $i=2,\dots, n$. 

In the first case $u(s)<-\e$. In the second case $u(s)<0$ by subharmonicity of $u$ in $B_i$, $i=2, \dots, n$.

So $u(s)<0$ for all $s\in [r,x]\setminus E$. But point $0$ is exactly like that, so $u(0)<0$. But by definition of $u$ we have that $u(0)=0$. Having a contradiction we are done with the lemma.
\end{proof}

Now we can prove Theorem \ref{MBth}.

\begin{proof}
Let us write the sum (here $x\in \R_+$)
$$
\sum_{k\ge 1} G(c_k, x) = \sum_{k=1}^N G(c_k,x) + \sum_{k=N+1}^\infty \frac{G(c_k,x)}{G(c_k,0)} G(c_k, 0)=:I+II.
$$

Recall that $M(x)=\lim_{k\to\infty}\frac{G(c_k,x)}{G(c_k,0)}$.  By Lemma \ref{2G} it is a monotonically increasing limit. Hence, 
$$
II\le M(x) \sum_{k=N+1}^\infty G(c_k, 0) \le \e M(x)
$$
uniformly for all $x\in \R_+$ and for as small $\e$ as we wish, as soon as $N$ is sufficiently large. We used here the Blaschke  assumption $\sum_{k=1}^\infty G(c_k, 0) <\infty$ of the theorem.

Having $N$ fixed we just notice that for each $k=1,\dots, N,$ we have $G(c_k,x) /M(x)\to 0$ as $x\to +\infty$ just by assumption \eqref{MG}. Theorem is proved.

\end{proof}

Recall that domain $\Omega=\C\setminus E, E\subset R,$ is called Widom domain if 
$$
\sum_{c: \nabla G(c, 0)=0} G(c, 0)<\infty.
$$
For Widom domains function $W(z):=\sum_{c: \nabla G(c, 0)=0} G(c, z)$ is called Widom function. It plays an important part in Hardy space theory in such domains. Widom domains automatically have property \eqref{MG}.

\begin{cor}
\label{widom}
Let $\Omega$ be a Widom domain, $E\subset \R_-$. Let $W(z):=\sum_{c: \nabla G(c, 0)=0} G(c, z)$.
Then $W(x)= o(M(x)), \, x\to +\infty$. Moreover, for any Blaschke product $B_\Omega$ all of whose zeros lie in $\Omega\cap\R_-$ we have
$$
\log|B_\Omega(x)| = o(M(x)),\, x\to +\infty.
$$
\end{cor}

\begin{cor}
\label{ha}
For any function $f$ of bounded characteristic in domain $\Omega=\C\setminus E, E\subset \R_-,$ such that
all its zeros lie in $\R_-$ and such that \eqref{MG} is satisfied, we have $h=a$, where $a$ is the number in the factorization Theorem \ref{fact}, and
$$
h=\limsup_{x\to +\infty} \frac{\log |f(x)|}{M(x)}.
$$
\end{cor}

\section{Another proof of Lemma \ref{2G}. Abelian integrals}
\label{abint}

We want to give another proof of this lemma based on  Abelian integrals, see Subsection \ref{comb}. Again we assume first (it is enough) that $E$ is a  subset of $\R_-$, say $E\subset (-\infty, a_0]$, $a_0<0$, and $E$ consists of finitely many  non-trivial intervals, so
$$
E= (-\infty, a_0]\setminus \bigcup_{k=1}^g (a_k, b_k).
$$
Let $b\in (a_{k_1},b_{k_1})$ and $a\in (a_{k_2},b_{k_2})$.
By \eqref{htog}
\begin{equation}\label{defgh}
\frac{G(z,b)}{G(0,b)}-\frac{G(z,a)}{G(0,a)}=\Re H(z,a,b),
\end{equation}
where
$$
H(z,a,b)=\int_{b_{k_1}}^z\frac{P_{g+1}(\lambda)}{(\lambda-a)(\lambda-b)}\frac{d\lambda}{s(\lambda)}, s^2=(z-a_0)\prod
_{j=1}^g(z-a_j)(z-b_j),
$$
and $P_{g+1}$ is a certain polynomial of degree $g+1$. Due to the normalization condition \eqref{periods} each of $g-2$ gaps, complimentary to $(a_{k_1},b_{k_1})$ and $(a_{k_2},b_{k_2})$, contains at least one zero $\xi_j$ of $P_{g+1}$. Moreover, since $\Re\, H(0,a,b)=0$, each of intervals $(a_0,0)$ and $(0,\infty)$ contains also  a zero of this polynomial. Thus, we were able to localize
$g$ zeros of $P_{g+1}(z)$. We claim that the remaining (with necessity real) zero $\xi$ of $P_{g+1}(z)$ belongs to the interval $(b,a)$.  

Assume that it does not. Then the standard arguments, related to  integrals of Schwarz-Christoffel type, show that
$\Re\, H(z,a,b)>0$ for all $z\in \Omega\cap(a,b)$. Indeed, by the definition $H(z,a,b)=\Re\, H(z,a,b)>0$ in $(b,b_{k_1})$. Then, if $z$ goes from $b$ to $a$ along the real line,  its image $w(z)=H(z,a,b)$ goes along straight lines with rotations by $\pi/2$ at the images of the end points $a_j$ and $b_j$ and by $\pi$ at $\xi_j\in(a_j,b_j)$. But, due to our assumption $\xi\not\in(b,a)$, we listed \textit{all}  switch-argument-points in this interval. Since  $w(z)$ remains to be pure imaginary at all points $z\in E\cap (a,b)$,  the integral $\Re\, w(z)$ has to be  positive in all gaps, including the interval 
$(a_{k_2},a)$.  That is,
$
\lim_{z\to a-0}\Re\, H(z,a,b)=+\infty
$
and this contradicts the definition \eqref{defgh}, according to which $\lim_{z\to a-0}\Re \, H(z,a,b)=-\infty$.

Thus $\xi\in (a,b)$. Then the same arguments, related to the Schwarz-Christoffel type integrals, show that
$\Re\, H(z,a,b)<0$ in all gaps inside $(a,0)$ including $(a_0,0)$ and $\Re\, H(z,a,b)>0$ in $(0,\infty)$, as well as in all gaps in $(-\infty,b)$. The lemma is proved.

\bigskip

\noindent
{\bf Remark.} Using Abel integrals approach one can also give a proof of Lemma \ref{twopo}. It is essentially more simple  since it does not require specification of a position of an ``extra zero" $\xi$ (the number of the critical zeros just coincides with the number of the specified intervals).


\begin{thebibliography}{ABCD}

\bibitem[Be]{Be}{M. Benedicks,} {\em Positive harmonic functions vanishing on the boundary of certain domains in $\R^n$}, Ark. Mat. {\bf 18} (1980), no.1, pp. 53--72.

\bibitem[dB]{dB}{L. de Branges, } Hilbert Spaces of Entire Functions. Prentice-Hall, Englewood Cliffs, N. J., 1968, ix+326 pp.

\bibitem[EYu]{EYu}{A. Eremenko, P. Yuditskii, }  {\em Comb functions.} Recent advances in orthogonal polynomials, special functions, and their applications, Contemp. Math., {\bf 578}, AMS, Providence, RI, 2012.

\bibitem[Ga]{Ga}{J. B. Garnett, } Bounded Analytic Functions. Revised first edition. Graduate Text in Mathematics, {\bf 236}, Springer, New York, 2007, xiv+459 pp.

\bibitem[JK]{JK}{D. Jerison, C. Kenig,} {\em Boundary behavior of harmonic functions in nontangentially accessible domains,} {\bf 46} (1982), no. 1, pp. 80--147.

\bibitem[Ha]{Ha}{W. Hayman,}  Meromorphic Functions. Oxford mathematical monographs, Clarendon Press, 1975, 195 pp.

\bibitem[Ko]{Ko}{P. Koosis, }  The Logarithmic Integral. I. Cambridge Studies in Advanced Mathematics, {\bf 12}, Cambridge University Press, 1988, xvi+606 pp.

\bibitem[L]{Lev}{N. Levinson, } Gap and Density Theorems. AMS Colloquium Publications, {\bf 26}, AMS, New York, 1940, viii+ 246 pp.

\bibitem[Ne]{Ne}{Nevanlinna, } Analytic Functions. Die Grundlehren  der mathematischen Wissenschaften, {\bf 162}, Springer-Verlag, New York-Berlin, 1970, viii+373 pp.

\bibitem[V]{V}{A. Volberg,} {\em On the dimension of harmonic measure on Cantor repellers}, Michigan Math. J., {\bf 40} (1993), no. 2, pp. 239--258.

\bibitem[Yu]{Yu}{P. Yuditskii,} On the $L^1$ extremal problem for entire functions.
 J. of Approx. Theory, 179(2014), 63-93.

\end{thebibliography}
\end{document}